\documentclass[11pt]{amsart}

\usepackage{amsmath}
\usepackage{fullpage}
\usepackage{xspace}
\usepackage[psamsfonts]{amssymb}
\usepackage[latin1]{inputenc}
\usepackage{graphicx,color}
\usepackage[curve]{xypic} 
\usepackage{hyperref}
\usepackage{graphicx}

\usepackage{amsmath}%
\usepackage{amsthm}%
\usepackage{amscd}
\usepackage{amsfonts}%
\usepackage{amssymb}%
\usepackage{graphicx}

\usepackage{mathrsfs}

\usepackage{tikz}
\usetikzlibrary{matrix,arrows}

\usepackage{tikz-cd}

\newtheorem{theorem}{Theorem}[section]
\theoremstyle{plain}

\newtheorem{conjecture}[theorem]{Conjecture}
\newtheorem{corollary}[theorem]{Corollary}
\newtheorem{lemma}[theorem]{Lemma}
\newtheorem{proposition}[theorem]{Proposition}

\theoremstyle{remark}

\newtheorem{example}[theorem]{Example}

\numberwithin{equation}{section}

\newcommand{\Kcal}{\mathscr{K}}

\newcommand{\Tcal}{\mathscr{T}}

\newcommand{\Z}{\mathbb{Z}}
\newcommand{\C}{\mathbb{C}}

\newcommand{\F}{\mathbb{F}}
\newcommand{\Q}{\mathbb{Q}}

\newcommand{\rad}{\mathrm{rad}}

\newcommand{\rk}{\mathrm{rk}}
\newcommand{\supp}{\mathrm{supp}}

\newcommand{\Der}{\mathrm{Der}}

\newcommand{\Hom}{\mathrm{Hom}}

\newcommand{\dd}{\mathrm{d}}

  \DeclareFontFamily{U}{wncy}{}
    \DeclareFontShape{U}{wncy}{m}{n}{<->wncyr10}{}
    \DeclareSymbolFont{mcy}{U}{wncy}{m}{n}
    \DeclareMathSymbol{\Sha}{\mathord}{mcy}{"58}

\begin{document}
\title[]{Arithmetic derivatives through geometry of numbers}

\author{Hector Pasten}
\address{ Departamento de Matem\'aticas,
Pontificia Universidad Cat\'olica de Chile.
Facultad de Matem\'aticas,
4860 Av.\ Vicu\~na Mackenna,
Macul, RM, Chile}
\email[H. Pasten]{hpasten@gmail.com}%

\thanks{This research was supported by ANID (ex CONICYT) FONDECYT Regular grant 1190442 from Chile.}

\date{\today}
\subjclass[2010]{Primary: 11J97; Secondary: 11H06, 14A23} %
\keywords{Arithmetic derivative, $abc$ Conjecture, Geometry of Numbers}%

\begin{abstract} We define certain arithmetic derivatives on $\mathbb{Z}$ that respect the Leibniz rule, are additive for a chosen equation $a+b=c$, and satisfy a suitable non-degeneracy condition. Using Geometry of Numbers, we unconditionally show their existence with controlled size. We prove that any power-saving improvement on our size bounds would give a version of the $abc$ Conjecture. In fact, we show that the existence of sufficiently small arithmetic derivatives in our sense is equivalent to the $abc$ Conjecture. Our results give an explicit manifestation of an analogy suggested by Vojta in the eighties, relating Geometry of Numbers in arithmetic to derivatives in function fields and Nevanlinna theory. In addition, our construction formalizes the widespread intuition that the $abc$ Conjecture should be related to arithmetic derivatives of some sort.
\end{abstract}

\maketitle

\setcounter{tocdepth}{1}

\tableofcontents


\section{Introduction}

\subsection{A map satisfying the Leibniz rule} There is great interest in constructing derivatives on $\Z$ behaving like derivatives on function fields, as they are expected to have remarkable applications. For instance, the arithmetic analogue of the Mason-Stothers  theorem is the $abc$ Conjecture, but the proof for polynomials heavily uses derivatives and  it is unclear how to adapt it to $\Z$. 

Let us discuss a first attempt by focusing only on the Leibniz rule. For each prime $p$ let $v_p$ denote the $p$-adic valuation on $\Q$ and let $\xi_p$ be a variable. Let $\Omega$ be the free $\Z$-module generated by the variables $\xi_p$. Let $\dd : \Z\to \Omega$ be the map defined by $\dd 0=0$ and by
$$
\dd n =n\sum_{p|n} \frac{v_p(n)}{p}\cdot \xi_p 
$$
for $n\ne 0$, where $p$ varies over the different prime divisors of $n$. (A version of  $\dd: \Z\to \Omega$ and generalizations can be found in \cite{KOW}.) Note that $n\cdot v_p(n)/p\in \Z$ when $p|n$, so $\dd n\in \Omega$ for all $n\in \Z$. In particular, when $p$ is prime we get $\dd p=\xi_p$. One immediately checks
\begin{lemma}[Leibniz rule for $\dd$]\label{LemmaLeibnizd} For all $a,b\in \Z$ we have $\dd(ab)=a\dd b + b\dd a$.
\end{lemma}
In fact, there is a sense in which $\dd:\Z\to \Omega$ is the universal map on $\Z$ satisfying the Leibniz rule, see Section \ref{SecFinal}. Unfortunately, this map $\dd$ is not a good analogue of a derivative because it is not additive: For instance, $\dd(1)=0$, $\dd(2)=\xi_2$, and $\dd(3)=\xi_3$ but we certainly have $0+\xi_2\ne \xi_3$.

\subsection{Arithmetic derivatives}  The starting point of our work is the following suggestion due to Thanases Pheidas: When  derivatives are applied in function field arithmetic, it is often the case that additivity is only needed finitely many times. Thus, one might still assign values to the variables $\xi_p$  in order to make $\dd$ additive  in the finitely many needed cases. For instance, in our previous example we may replace $\xi_2$ and $\xi_3$ by $1$ to get $0+1=1$ from the equation $1+2=3$.

Our aim is to investigate this construction in the simplest non-trivial case: When exactly one additive condition is imposed. For this, it is convenient to give an algebraic formulation of Pheidas's suggestion.

Consider a group morphism $\psi:\Omega\to \Z$.  The \emph{arithmetic derivative} $\dd^\psi$ attached to  $\psi$ is simply defined as $d^\psi=\psi\circ \dd:\Z\to \Z$. Note that  $\dd^\psi:\Z\to \Z$ still respects the Leibniz rule.

Given coprime positive integers $a,b,c$ with $a+b=c$, the condition $\dd^\psi (a)+\dd^\psi (b) = \dd^\psi (c)$ imposes a linear equation on the values $\psi(\xi_p)$. When $c>2$, the set of all such maps $\psi$ satisfying $\psi(\xi_p)=0$ whenever $p\nmid abc$, turns out to be a non-trivial free abelian group, cf. Lemma \ref{LemmaExistence}. We denote this group by $\Tcal(a,b)$. With this notation, one can ask to what extent an arithmetic derivative $\dd^\psi$ for $\psi\in \Tcal(a,b)$ can be used to mimic arguments from function field arithmetic.


\subsection{The Small Derivatives Conjecture} Let us focus our attention on a particular kind of morphism $\psi:\Omega\to \Z$. For us, a \emph{derivation} is a group morphism $\psi:\Omega\to \Z$ satisfying that its norm $\|\psi\|:=\sup_p |\psi(\xi_p)|$ is finite. The set of all such maps is a $\Z$-module denoted by $\Tcal$ which comes equipped with the norm $\|-\|$. The previously defined groups $\Tcal(a,b)$ are contained in $\Tcal$.

Besides these definitions, we also introduce the notion of \emph{$\psi$-independence} for a pair of integers $(a,b)$ and a derivation $\psi$, by requiring that the \emph{arithmetic Wronskian} $W^\psi(a,b)=a\dd^\psi b-b\dd^\psi a$ is non-zero.  Our study focuses on the question of existence of small (in the sense of $\|-\|$) derivations $\psi\in \Tcal(a,b)$ satisfying that $a,b$ are $\psi$-independent.  We propose the following
\begin{conjecture}[Small Derivatives Conjecture, cf. Conjecture \ref{ConjSDC}] There is an absolute constant $0<\eta<1$ such that for all but finitely many triples of coprime positive integers $(a,b,c)$ satisfying $a+b=c$ and not of the form $(1,N, q)$  with $q$ prime (up to order), the following holds: There is $\psi\in \Tcal(a,b)$ such that $a,b$ are $\psi$-independent and $\|\psi\|  < c^{\eta}$.
\end{conjecture}
This conjecture seems to capture the usefulness of derivatives in function field arithmetic in the sense that it allows one to translate arguments from function fields to $\Z$, provided that additivity of derivatives is used just once. In order to clarify how to use our arithmetic derivatives together with the Small Derivatives Conjecture to perform such a translation, in Section \ref{SecFLT} we give a short proof of the analogue of Fermat's Last Theorem for $\C[x]$ based on derivatives without using the Mason-Stothers theorem or radicals, and then we translate the argument to $\Z$. We conclude that the Small Derivatives Conjecture implies the asymptotic form of Fermat's Last Theorem.

The connection with Fermat's Last Theorem is of course just an example to clarify the analogy between our arithmetic derivatives and the usual function field derivatives. Actually, our main goal is to show that the Small Derivatives Conjecture is equivalent to the $abc$ Conjecture (with a suitable choice of exponents). Let us give a brief outline of the main results.


\subsection{Main results}
In Theorem \ref{ThmExistence} we will use geometry of numbers to show that $\Tcal(a,b)$ admits a full set of linearly independent derivations with controlled norm. In Theorem \ref{Thmabc} we prove an unconditional $abc$-type  bound which explicitly includes a contribution coming from the norm of arithmetic derivatives.  This motivates the problem of producing $\psi\in \Tcal(a,b)$ for a given pair of coprime positive integers $(a,b)$ such that $\|\psi\|$ is small and $a,b$ are $\psi$-independent. We prove such a result in Lemma \ref{LemmaSmall} but unfortunately it is insufficient to prove the $abc$ Conjecture. Nevertheless, this analysis motivates a heuristic (cf. Section \ref{SecHeuristic}) leading to the formulation of the Small Derivatives Conjecture discussed above. As for evidence, besides Lemma \ref{LemmaSmall} and the heuristic in Section \ref{SecHeuristic}, we prove a  version of the Small Derivatives Conjecture with exponent $\eta=1/2 +\epsilon$, provided that the $\psi$-independence condition is replaced by a somewhat weaker non-degeneracy condition, see Theorem \ref{ThmNonZero}.

Our main results concerning the arithmetic relevance of these notions are Lemma \ref{LemmaEquiv} and Theorem \ref{ThmEquiv}; see also Corollary \ref{CoroEquiv}. These results show  that the Small Derivatives Conjecture is equivalent to the $abc$-conjecture, with a precise dependence of exponents. 


\subsection{Some algebraic context} In Section \ref{SecFinal} we include a discussion on a generalization of the constructions $\Omega$ and $\Tcal$ from an algebraic point of view. Consider a commutative monoid $R$, a commutative unitary ring $A$, and a morphism of monoids $\alpha:R\to A$ where $A$ is taken as a multiplicative monoid. For an $A$-module $U$, we say that a map $D:R\to U$ is an \emph{$\alpha$-derivation} (with values in $U$) if $D(\alpha(r))=0$ for every $r\in R$ and $D(ab)=aD(b)+bD(a)$ for all $a,b\in A$.

We will construct a universal $\alpha$-derivation $\dd_{(A,\alpha)}: A\to \Omega_{(A,\alpha)}$ and compute it in some examples. One of these examples shows that our map $\dd:\Z\to \Omega$ is precisely the universal $\alpha$-derivation on $\Z$ for the inclusion map $\alpha:\{-1,1\}\to \Z$. So, in this sense, the map $\dd:\Z\to \Omega$ is not artificial.  

Our notion of $\alpha$-derivations is very similar to the theory of absolute derivations from \cite{KOW}, except that we keep track of the additional data of a morphism of monoids $\alpha:R\to A$ ---in fact, when $R=\{1\}$ we recover the absolute derivations from \cite{KOW}. 

The additional data of a morphism of monoids is natural from various points of view. First, in our arithmetic applications it corresponds to restricting the support of the derivations $\psi\in \Tcal$, which was necessary in the definition of $\Tcal(a,b)$. Secondly, one can check compatibility with localization of our $\alpha$-derivations, leading to sheaves of $\alpha$-derivations on pre-log schemes (although we do not pursue this direction in this work). From this point of view, our modules $\Tcal(a,b)$ give normed sheaves on $\mathrm{Spec}(\Z)$ endowed with a suitable pre-log structure. Finally,  monoids are often considered as the most basic ``ground field'' in the general $\F_1$ philosophy, which motivates the construction of derivatives on $\Z$ by requiring compatibility with  monoids rather than requiring linearity.


\subsection{Remarks on arithmetic derivatives}  In summary, this work formalizes the widespread intuition that some sort of arithmetic derivative on $\Z$ should be closely related to the $abc$ Conjecture.  Our results  are in line with Vojta's proposed analogy comparing Geometry of Numbers in arithmetic to derivatives in the setting of function fields and Nevanlinna theory; see Chapter 6 in \cite{VojtaThesis}. We stress the fact that ---despite the close relation with more sophisticated concepts such as ``geometry over $\F_1$''--- our constructions only involve classical tools.

It is worth pointing out that Vojta has a different proposal for arithmetic derivatives in terms of the existence of small rational points in the total space of certain projective bundles (the \emph{Tautological Conjecture}, cf. Section 30 in \cite{VojtaCIME}). Also, Faltings \cite{FaltingsKS} investigated yet another possible notion of arithmetic derivative in terms of certain axiomatically defined arithmetic analogue of the Kodaira-Spencer class for fibrations, showing that such an object cannot exist.

Finally, we mention that Buium (see \cite{Buium} and the references therein) developed a theory of $p$-derivations, which affords some analogies between differential calculus and the arithmetic of local fields. Buium's $p$-derivations, however, are purely local and they do not seem to be related to the global notion of arithmetic derivative in the present work.


\section{Derivations and arithmetic derivatives}

\subsection{The module $\Tcal$ and arithmetic derivatives} Recall (from the Introduction) that $\Omega$ is the free $\Z$-module generated by the variables $\xi_p$ for $p$ varying over prime numbers. For a $\Z$-linear map $\psi:\Omega\to \Z$ we define $\|\psi\|  = \sup_p |\psi(\xi_p)|$. We will often  use the observation that if $\psi\ne 0$ then $\|\psi\|\ge 1$. Let
$$
\Tcal = \{\psi\in \Hom_\Z(\Omega,\Z) : \|\psi\| \mbox{ is finite}\}.
$$
Elements of $\Tcal$ will be called \emph{derivations} and $\|-\| $ is a norm on the $\Z$-module $\Tcal$. 

Given a derivation $\psi\in \Tcal$ we define the \emph{arithmetic derivative} attached to $\psi$ as the map
$$
\dd^\psi:\Z\to \Z \quad\mbox{defined by } \quad \dd^\psi:=\psi\circ \dd.
$$

For example, the classical ``arithmetic derivative'' that one encounters in elementary number theory \cite{MingotShelly, Barbeau} is precisely $\dd^{\sigma}$ where $\sigma(\sum_p a_p\xi_p)=\sum_p a_p$  ---note that $\|\sigma\| =1$ so $\sigma\in \Tcal$.

Returning to the general case, observe that upon composing with $\psi\in \Tcal$, Lemma \ref{LemmaLeibnizd} gives
\begin{lemma}[Leibniz rule for arithmetic derivatives]\label{LemmaLeibnizAD} Let $\psi\in \Tcal$. For every $a,b\in \Z$ we have $\dd^\psi(ab)=a\dd^\psi b +b\dd^\psi a$. Thus, for all integers $n\ge 1$ and all $a\in \Z$ we have $\dd^\psi(a^n)=na^{n-1}\dd^\psi a$.
\end{lemma}

Concerning norms, the following estimates are useful.

\begin{lemma}\label{Lemmavbd} For every positive integer $n$ we have $\sum_{p|n}v_p(n)/p\le (2\log 2)^{-1} \log n$. In particular, if $n\ge 2$ and $\psi\in \Tcal$,  then $|\dd^\psi(n)| < \|\psi\|\cdot n\log n$.
\end{lemma}
\begin{proof} We can assume $n\ge 2$. Then we get
$$
\sum_{p|n}\frac{v_p(n)}{p}=\sum_{p|n}v_p(n)\log p\cdot \frac{1}{p\log p}\le \left(\max_{p|n}\frac{1}{p\log p}\right)\log n\le \frac{\log n}{2\log 2}.
$$
The last claim is immediate from $\dd^\psi(n)=n\sum_{p|n} v_p(n)p^{-1} \psi(\xi_p)$.
\end{proof}

\subsection{The modules $\Tcal(a,b)$} The \emph{support} of  $\psi\in \Tcal$ is the set of primes $\supp(\psi)=\{p : \psi(\xi_p)\ne 0\}$. The support of a non-zero integer $n$ is $\supp(n)=\{p : p|n\}$ and the number of different prime factors is $\omega(n)=\# \supp(n)$. We recall  the following elementary fact:

\begin{lemma}\label{Lemmaw} We have $\omega(n)=O(\log(n)/\log \log n)$. In particular, for each $\epsilon>0$ we have the bound  $\omega(n)<\epsilon \log n$ for all but finitely many positive integers $n$.
\end{lemma}

 For a pair of positive integers $a,b$ we define
$$
\Tcal(a,b)=\{\psi\in \Tcal : \supp(\psi)\subseteq \supp (ab(a+b))\mbox{ and }\dd^\psi(a+b)=\dd^\psi a+\dd^\psi b\}
$$
(since $a$ and $b$ are positive, $\supp (ab(a+b))$ is a finite set.) Thus, for $\psi\in \Tcal(a,b)$ we have that the arithmetic derivative $\dd^\psi$ not only satisfies the Leibniz rule, but also, it satisfies $\dd^\psi(a+b)=\dd^\psi a+\dd^\psi b$ for the chosen integers $a$ and $b$. Explicitly, the condition $\dd^\psi(a+b)=\dd^\psi a+\dd^\psi b$ is
\begin{equation}\label{EqnAdd}
a\sum_{p|a}\frac{v_p(a)}{p}\cdot \psi(\xi_p)+b\sum_{p|b}\frac{v_p(b)}{p}\cdot  \psi(\xi_p) = (a+b)\sum_{p|a+b}\frac{v_p(a+b)}{p}\cdot  \psi(\xi_p)
\end{equation}
which is a homogeneous linear equation on the unknowns $ \psi(\xi_p)$ for $p\in \supp(ab(a+b))$. Hence:
\begin{lemma}[Basic existence lemma]\label{LemmaExistence} Let $a$ and $b$ be positive integers. Then $\Tcal(a,b)$ is a saturated $\Z$-submodule of $\Tcal$ of rank $\omega(ab(a+b))-1$. 
\end{lemma}

\subsection{Bounding the norm} We aim for a more precise version of Lemma \ref{LemmaExistence}. First, we note that for all $m,n,k\in \Z$ we have
$$
\dd (km+kn)-\dd(km) - \dd(kn)=k\cdot \left(\dd(m+n)-\dd m -\dd n\right)
$$
and similarly for $\dd^\psi$ for any $\psi\in \Tcal$. Hence, the question of existence of arithmetic derivatives respecting additivity for a chosen pair of numbers can be reduced to the coprime case.

We will need the following version of Siegel's lemma which builds on Minkowski's second theorem in Geometry of Numbers, see Theorem 2 in \cite{BombieriVaaler}.

\begin{theorem}[Siegel's lemma]\label{ThmSL} Let $a_1,...,a_N\in \Z$. The equation $a_1X_1+...+a_NX_N=0$ has linearly independent solutions ${\bf x}_i = (x_{i1},...,x_{iN})\in \Z^N$ for $1\le i\le N-1$ satisfying  
$$
\prod_{i=1}^{N-1} \max_{1\le j\le N} |x_{ij}|\le N\cdot \max_{1\le j\le N}|a_j|.
$$
\end{theorem}

With this at hand, we can prove a more precise version of Lemma \ref{LemmaExistence}, which we state in the case of positive integers for the sake of simplicity.

\begin{theorem}[Existence of arithmetic derivatives of controlled size] \label{ThmExistence} 
Suppose that $a,b$ are coprime positive integers with $c:=a+b>2$, i.e., $(a,b)\ne (1,1)$. Then  $\Tcal(a,b)$ has rank $r:=\omega(abc)-1\ge 1$ and  there are $\Z$-linearly independent derivations $\psi_1,...,\psi_r\in \Tcal(a,b)$ satisfying
$$
\prod_{i=1}^r \|\psi_i\|  \le \frac{\omega(abc)}{2\log 2}\cdot c\log c. 
$$
\end{theorem}
\begin{proof} As in \eqref{EqnAdd}, the condition  $\dd^{\psi} a+ \dd^{\psi} b = \dd^{\psi} c$ defining $\Tcal(a,b)$ becomes
$$
a\sum_{p|a}\frac{v_p(a)}{p}\cdot \psi(\xi_p)+b\sum_{p|b}\frac{v_p(b)}{p}\cdot  \psi(\xi_p) = c\sum_{p|c}\frac{v_p(c)}{p}\cdot  \psi(\xi_p)
$$
Since $(a,b)\ne (1,1)$ we have $r\ge1$. Treating $\psi(\xi_p)$ as unknowns and using the fact that $a$, $b$, and  $c$ are pairwise coprime, the coefficients of the previous equation are positive integers bounded by $c\log_2(c)/2$ where $\log_2$ is the base $2$ logarithm. The result follows by Theorem \ref{ThmSL}.
\end{proof}
Choosing the smallest derivation provided by the previous theorem one deduces:
\begin{corollary}[Existence of a small derivative]\label{CoroExistence} Let $\epsilon>0$. For all but finitely many triples of coprime integers $a,b,c$ with $c>2$ and satisfying $a+b=c$, there is a non-zero $\psi\in \Tcal(a,b)$ with $\|\psi\|< c^{\frac{1}{r} +\epsilon}$, where $r=\omega(abc)-1$.
\end{corollary}
However, Corollary \ref{CoroExistence} does not ensure any sort of non-degeneracy for the arithmetic derivative it provides. For instance, although $\psi$ is not zero, it can occur that $\dd^\psi(a)=\dd^\psi(b)=\dd^\psi(c)=0$. The following result  remedies this situation.

\begin{theorem}[Small non-trivial derivatives]\label{ThmNonZero} Let $\epsilon>0$. For all but finitely many triples of coprime integers $a,b,c$ larger than $1$ that satisfy $a+b=c$, there is $\psi\in \Tcal(a,b)$ with $\|\psi\|< c^{\frac{1}{2} +\epsilon}$ such that not all the integers $\dd^\psi(a)$, $\dd^\psi(b)$, $\dd^\psi(c)$ are zero.
\end{theorem}
\begin{proof} Since $a,b,c$ are larger than $1$, each one of them has prime divisors. Thus, the conditions \eqref{EqnAdd}, $\dd^\psi(a)=0$, and $\dd^\psi(b)=0$ are linearly independent when we consider the terms $\psi(\xi_p)$ as unknowns. Let $\Kcal(a,b)\subseteq \Tcal(a,b)$ be the subgroup defined by these conditions and note that $\rk \Kcal(a,b)=r-2$ where $r=\rk \Tcal(a,b)=\omega(abc)-1$, see Lemma \ref{LemmaExistence}.

Let $\psi_1,...,\psi_r\in \Tcal(a,b)$ be as provided by Theorem \ref{ThmExistence} and assume that they are labeled in such a way that $\|\psi_1\|\le \|\psi_2\|\le ...\le \|\psi_r\|$. Since the $\psi_i$ are linearly independent, there are indices $i_1<i_2$ such that $\psi_{i_1}$ and $\psi_{i_2}$ are not in $\Kcal(a,b)$. Then we have
$$
\|\psi_{i_1}\|^2\le \|\psi_{i_1}\|\cdot \|\psi_{i_2}\|\le \prod_{i=1}^r \|\psi_i\|\le \frac{\omega(abc)}{2\log 2}\cdot c\log c.
$$
and we conclude by Lemma \ref{Lemmaw}.
\end{proof}

We will be interested in a more delicate notion of non-degeneracy for a derivation $\psi\in \Tcal(a,b)$, for which we need to introduce certain arithmetic Wronskians.

\subsection{Independence}

One might be tempted to explore analogues of various notions from differential calculus using the functions $\dd^\psi:\Z\to\Z$ instead of an actual derivative. Rather than giving a lengthy list of such definitions, let us simply mention here a notion that will be useful for us. Given $\psi\in \Tcal$, the \emph{$\psi$-Wronskian} of two integers $a,b$ is defined by
$$
W^\psi(a,b)=\det\left[
\begin{array}{cc}
a&b\\
\dd^\psi a& \dd^\psi b
\end{array}
\right]=a\dd^\psi b-b\dd^\psi a \in\Z.
$$
Let us also note the formula
\begin{equation}\label{EqnW}
W^\psi(a,b)=ab\cdot \left(\sum_{p|b}\frac{v_p(b)}{p}\psi(\xi_p) - \sum_{p|a}\frac{v_p(a)}{p}\psi(\xi_p)\right).
\end{equation}
We say that $a,b$ are \emph{$\psi$-dependent} if $W^\psi(a,b)= 0$. Otherwise, they are \emph{$\psi$-independent}. From \eqref{EqnW} we deduce  that $a,b$ are $\psi$-dependent if and only if
\begin{equation}\label{EqnDep}
\sum_{p|a}\frac{v_p(a)}{p}\psi(\xi_p) =  \sum_{p|b}\frac{v_p(b)}{p}\psi(\xi_p).
\end{equation}
Given positive integers $a$ and $b$ we define 
$$
\Tcal^\circ(a,b)=\{\psi\in \Tcal(a,b) : a,b \mbox{ are $\psi$-dependent}\}.
$$
\begin{lemma}\label{LemmaDep} Let $a,b$ be coprime positive integers with $(a,b)\ne (1,1)$. The set $\Tcal^\circ(a,b)$ is a saturated $\Z$-submodule of $\Tcal(a,b)$ with  $\rk \Tcal^{\circ}(a,b)= \rk \Tcal(a,b)-1=\omega(ab(a+b))-2$.  In particular, $\Tcal^\circ(a,b)$ is   properly contained in $\Tcal(a,b)$.
\end{lemma}
\begin{proof} Since $(a,b)\ne (1,1)$ there is some prime $q|ab$. Hence, the Equation \eqref{EqnDep} defining $\Tcal^\circ(a,b)$ is non-trivial. Furthermore, no term corresponding to primes $p|c$ contributes to \eqref{EqnDep}, while they appear in the Equation \eqref{EqnAdd} defining $\Tcal(a,b)$. This proves that, considering the values $\psi(\xi_p)$ as variables, the Equations \eqref{EqnAdd} and \eqref{EqnDep} are linearly independent. We conclude by Lemma \ref{LemmaExistence}.
\end{proof}

\section{An $abc$ bound and the problem of small arithmetic derivatives}

\subsection{The $abc$ Conjecture} The radical of a positive integer $n$, denoted by $\rad(n)$, is the product without repetitions of the different primes dividing $n$. The celebrated $abc$ Conjecture is
\begin{conjecture}[The Masser-Oesterl\'e $abc$ Conjecture]\label{Conjabc} Given $\epsilon>0$, there is a constant $\kappa_\epsilon>0$ such that for all coprime positive integers $a,b,c$ with $a+b=c$ we have $c<\kappa_\epsilon\cdot  \rad(abc)^{1+\epsilon}$.
\end{conjecture}

For many applications even the following weaker version would suffice:

\begin{conjecture}[Oesterl\'e's $abc$ Conjecture]\label{ConjWeakabc} There is an absolute constant $M$ such that for all coprime positive integers $a,b,c$ with $c=a+b$ we have $c<  \rad(abc)^{M}$.
\end{conjecture}

Oesterl\'e's version of the $abc$ Conjecture was proposed first in 1985, and it was later refined into the Masser-Oesterl\'e $abc$ Conjecture by Masser. See \cite{Masserabc} for a historical  account of how these conjectures were formulated. To the best of the author's knowledge, they remain open.

\subsection{An $abc$ bound using arithmetic derivatives} The notion of derivation considered in the previous section is enough to get an estimate in the spirit of the $abc$ Conjecture, with a proof  analogous to Snyder's proof of Mason's Theorem in the function field setting (see \cite{Snyder}) or to the proof of the Second Main Theorem in Nevanlinna theory using Wronskians or logarithmic derivatives. 
\begin{theorem}[An $abc$ estimate]\label{Thmabc} Let $a,b$ be coprime positive integers with $(a,b)\ne (1,1)$ and let $\psi\in \Tcal(a,b)$. Suppose that $a$ and $b$  are $\psi$-independent. Writing $c=a+b$, we have
$$
\frac{c}{\log c}\le  \rad(abc)\cdot \frac{\|\psi\| }{\log 2}.
$$
\end{theorem}
For the proof, we need a simple observation.

\begin{lemma}\label{Lemmarad} For any positive integer $n$ and any $\psi\in \Tcal$, we have that $n$ divides $\gcd(n, \dd^\psi n)\cdot\rad(n)$.
\end{lemma}
\begin{proof} $n$ divides $n\cdot \rad(n)$. From the definition of $\dd^\psi$, we see that $n$ divides $(\dd^\psi n )\cdot \rad(n)$.
\end{proof}
\begin{proof}[Proof of Theorem \ref{Thmabc}] The equation $\dd^\psi a+\dd^\psi b=\dd^\psi c$ gives
$$
W:=W^\psi(a,b)=W^\psi(a,c)=W^\psi(c,b)
$$
which is non-zero because $a,b$ are $\psi$-independent. By Lemma \ref{Lemmarad}, we see that $a/\rad(a)$ divides $W=W^\psi(a,b)$, and similarly for $b$ and $c$. By coprimality of $a$, $b$, and $c$ we get that $abc$ divides $W\cdot \rad(abc)$. Since $W\ne 0$, we conclude $abc\le |W|\cdot \rad(abc)$. From \eqref{EqnW} we deduce
$$
\frac{abc}{\rad(abc)}\le |W|=ab \left|\sum_p\frac{v_p(a)}{p}\psi(\xi_p)-\sum_p\frac{v_p(b)}{p}\psi(\xi_p)\right|\le ab\|\psi\|  \sum_{p|ab}\frac{v_p(ab)}{p}\le ab\|\psi\| \cdot  \frac{\log (ab)}{2\log 2}
$$
where the last bound is by Lemma \ref{Lemmavbd}. The result follows from $\log(ab)\le 2\log c$.
\end{proof}

\subsection{Small arithmetic derivatives}\label{SecHeuristic}

In view of Theorem \ref{Thmabc}, we cannot avoid the question of existence of small derivations $\psi\in \Tcal(a,b)$ subject to the condition that $a,b$ be $\psi$-independent. A first result is directly deduced from  Lemma \ref{LemmaDep} and Theorem \ref{ThmExistence}.

\begin{lemma}[Small arithmetic derivatives satisfying independence]\label{LemmaSmall} Let $a,b$ be coprime positive integers with $(a,b)\ne (1,1)$ and let $c=a+b$. Let $r=\omega(abc)-1$ and note that $r\ge 1$. For any list of linearly independent derivations $\psi_1,...,\psi_r\in \Tcal(a,b)$ there is at least one index $1\le i_0\le r$ such that $a,b$ are $\psi_{i_0}$-independent. Furthermore, choosing $\psi_1,...,\psi_r$ as in Theorem \ref{ThmExistence} we get
$$
\|\psi_{i_0}\| \le \frac{\omega(abc)}{2\log 2} \cdot c\log c.
$$
\end{lemma}

\begin{example} Let $q=2^n-1$ be a Mersenne prime and take $a=1$, $b=q$, $c=2^n$. Then $\Tcal(1,q)=\Z\cdot \psi_1$ where the $\psi_1(\xi_2)=1$, $\psi_1(\xi_q)=n\cdot 2^{n-1}$, and $\psi_1(p)=0$ for all $p\ne 2,q$. Thus, in this example the bound given by Lemma \ref{LemmaSmall} is sharp up to a factor of $2$, because we actually have:
$$
\|\psi_1\| = n\cdot 2^{n-1} = \frac{\omega(abc)}{4\log 2}\cdot c\log c. 
$$
\end{example}

Unfortunately, Lemma \ref{LemmaSmall} combined with Theorem \ref{Thmabc} falls short of proving the $abc$ Conjecture. Nevertheless, it clarifies the fact that in order to prove the $abc$ Conjecture one must get a power-saving improvement over the bound in Lemma \ref{LemmaSmall}.

Optimistically, we may expect that in Theorem \ref{ThmExistence} one can choose the $\psi_i$ such that all the $\log \|\psi_i\| $ have roughly the same size. Proceeding as in Lemma \ref{LemmaSmall}, if $\omega(abc)\ge 3$ (i.e. $r\ge 2$) this would give the desired power-saving improvement. Regarding the condition $\omega(abc)\ge 3$, we have:
\begin{lemma}\label{LemmaCatalan} Up to order, the only triples of  coprime positive integers $a,b,c$ with $a+b=c$ having $\omega(abc)\le 2$ are the following: $(1,1,2)$, $(1,8,9)$, and $(1,2^n,q)$ with $q$ prime and $n\ge 1$. 
\end{lemma}
This follows from Mihailescu's theorem \cite{Mihailescu}. Of course, it is not known  whether there are infinitely many primes of the form $q=2^n+1$ (Fermat primes) or $q=2^n-1$ (Mersenne primes).

There is, however, an additional caveat in the previous heuristic. If $a,b,c$ are, up to order, $1,q,N$ for some prime $q$, then from the defining equations \eqref{EqnAdd} and \eqref{EqnDep} we see that every $\psi\in \Tcal^\circ(a,b)$ satisfies the unexpected condition $\psi(\xi_q)=0$. If in addition $N$ is the product of powers of small primes, then it can happen that $\Tcal^\circ(a,b)$ is generated by unusually small derivations, in which case our heuristic justification on how to get a power-saving improvement over Lemma \ref{LemmaSmall} fails.

\begin{example} Consider $a=1$, $b=108=2^2\cdot 3^3$, and $c=q=109$. Then $r=2$ and the group $\Tcal^0(1,108)\simeq \Z$ is generated by the derivation $\psi_1$ determined by $(\psi_1(2),\psi_2(3),\psi_3(109)) = (1,-1,0)$. On the other hand, any derivation $\psi_2\in \Tcal(1,108)$ which is linearly independent from $\psi_1$ satisfies $\|\psi_2\|\ge 108$, with equality achieved (for instance) at $(\psi_2(2),\psi_2(3),\psi_2(109)) = (2,-1,108)$.
\end{example}

The previous considerations motivate our main conjecture:

\begin{conjecture}[Small Derivatives Conjecture]\label{ConjSDC} There is an absolute constant $0<\eta<1$ such that for all but finitely many triples of coprime positive integers $(a,b,c)$ satisfying $a+b=c$ and not of the form $(1,N, q)$  with $q$ prime (up to order), the following holds: There is $\psi\in \Tcal(a,b)$ such that $a,b$ are $\psi$-independent and $\|\psi\|  < c^{\eta}$.
\end{conjecture}

  The crucial aspects of Conjecture \ref{ConjSDC} are that the exponent $\eta$ is strictly less than $1$ and that $a,b$ must be $\psi$-independent.  Some of our results provide unconditional evidence:
\begin{itemize}

\item Corollary \ref{CoroExistence} shows that if we completely drop the $\psi$-independence condition, then the desired bound holds for any  $\eta>0$, for those triples $a,b,c$ satisfying $\omega(abc)>1+1/\eta$.

\item Theorem \ref{ThmNonZero} shows that if we replace the $\psi$-independence condition by the weaker requirement that $\dd^\psi(a)$ or $\dd^\psi(b)$ be non-zero, then one can indeed achieve a bound with exponent $\eta<1$ ---in fact, any $\eta>1/2$ works. (Note that if $\psi\in \Tcal(a,b)$ and $a,b$ are $\psi$-independent, then necessarily $\dd^\psi(a)$ or $\dd^\psi(b)$ is non-zero.) 
 
\item  Lemma \ref{LemmaSmall} shows that if we keep the $\psi$-independence condition, then a version of the Small Derivatives Conjecture holds with exponent $\eta=1+\epsilon$ rather than the sought  $\eta<1$.

\end{itemize}

\subsection{Proof of concept: Fermat's Last Theorem}\label{SecFLT} As it is well-known, the analogue of Fermat's Last Theorem (FLT) over polynomials can be deduced from the Mason-Stothers theorem, and the same argument over $\Z$ shows that the $abc$ Conjecture implies the ``asymptotic'' FLT, meaning FLT up to finitely many exponents (of course, FLT was proved by Wiles \cite{Wiles}, while the $abc$ Conjecture remains open.) Let us give a direct  proof\footnote{We make no claim of originality on this argument, although we could not find it in the literature.} of FLT for the polynomial ring $\C[x]$ without using the Mason-Stothers theorem or radicals. Recall that the Wronskian of  $f,g\in \C[x]$ is $W(f,g)=fg'-f'g$. 

\begin{proposition}[FLT for polynomials] Let $n \ge 3$. Let $f,g,h\in \C[x]$ be coprime non-zero polynomials with at least one of them non-constant. Then $f^n+g^n\ne h^n$. 
\end{proposition}
\begin{proof} For the sake of contradiction, suppose that  $f^n+g^n=h^n$. Without loss of generality, assume that $h$ has the largest degree among $f,g,h$.  Note that $W(f,h)\ne 0$, for otherwise we would have $f=\lambda h$ and $g=(1-\lambda)h$ for some $\lambda\in \C$, which is not possible. 

Taking derivatives and multiplying by $f$ we find $f^{n}f'+fg^{n-1}g' = fh^{n-1}h'$. Using $f^nf'=(h^n-g^n)f'$ we get $g^{n-1}W(f,g)=h^{n-1}W(f,h)$. Since $W(f,h)\ne 0$ and $g,h$ are coprime, we find 
$$
(n-1)\deg(h)\le \deg W(f,g) \le \deg(fg)-1< 2\deg(h) 
$$
which implies $n<3$; contradiction.
\end{proof}

Our theory of arithmetic derivatives affords a smooth translation of the previous proof into the setting of integers, conditional on the Small Derivatives Conjecture \ref{ConjSDC}.
\begin{proposition}[Asymptotic FLT conditional on the Small Derivatives Conjecture] \label{PropFLTAD} Assume Conjecture \ref{ConjSDC}. There is a positive integer $n_0$ such that for all $n\ge n_0$ the following holds: If $a,b,c$ are coprime positive integers, then $a^n+b^n\ne c^n$. 
\end{proposition}
\begin{proof} Assume Conjecture \ref{ConjSDC} with some exponent $\eta<1$ and let $n\ge 2$ be a positive integer. Thus, for all but finitely many triples of coprime integers $a,b,c$ with $a^n+b^n=c^n$ there is $\psi\in \Tcal(a^n,b^n)$ such that $\|\psi\| < c^{n\cdot \eta}$ and $W^\psi(a^n,b^n)\ne 0$ ($a^n,b^n,c^n$ are not prime).  Note that $\dd^\psi(a^n)=na^{n-1}\dd^\psi a$ by Lemma \ref{LemmaLeibnizAD} and similarly for $b$, so $W^\psi(a^n,b^n)=n(ab)^{n-1}W^\psi(a,b)$, concluding $W^\psi(a,b)\ne 0$. 

Starting from $a^n+b^n=c^n$ we repeat the computation from the polynomial case using  Lemma \ref{LemmaLeibnizAD} and the fact that $\psi\in \Tcal(a^n,b^n)$. We get $b^{n-1}W^\psi(a,b)=c^{n-1}W^\psi(a,c)$. Since $W^\psi(a,b)\ne 0$ and $b,c$ are coprime, Lemma \ref{Lemmavbd} yields
$$
c^{n-1}\le |W^\psi(a,b)| = |a\dd^\psi b - b\dd^\psi a|< \|\psi\| \cdot 2c^2\log c<2c^{2+n\cdot \eta}\log c
$$
Up to finitely many triples $(a,b,c)$ this shows $n\le  3/(1-\eta)$, which suffices to prove the result.
\end{proof}

In  Section \ref{SecEquiv} we will show that the Small Derivatives Conjecture is equivalent to the $abc$ Conjecture and, in this way, one can prove Proposition \ref{PropFLTAD} by using the $abc$ Conjecture as an intermediate step. Nevertheless, the previous proof gives an example of how to use our arithmetic derivatives to directly translate arguments from function field arithmetic to the integers. 


\section{Small arithmetic derivatives are equivalent to the $abc$ Conjecture}\label{SecEquiv}

\subsection{The Small Derivatives Conjecture implies the $abc$ Conjecture}

\begin{lemma}\label{LemmaEquiv}
If the Small Derivative Conjecture \ref{ConjSDC} holds for some value of $\eta$, then Oesterl\'e's $abc$ Conjecture \ref{ConjWeakabc} holds for every $M>1/(1-\eta)$.
\end{lemma}
\begin{proof} Assume Conjecture \ref{ConjSDC} for some exponent $0<\eta<1$. If up to order we have $(a,b,c)=(1,N,q)$ with $q$ prime and $N\ge 2$, then $\rad(abc)\ge 2q> q+1\ge c$, hence, the $abc$ Conjecture holds in such cases. So, we may assume we are not in the previous case. For all but finitely many triples of coprime positive integers $a,b,c$ with $a+b=c$ we have
$$
\frac{c}{\log c} < \rad(abc)\cdot \frac{c^\eta}{\log 2}
$$
where we applied Theorem \ref{Thmabc} and Conjecture \ref{ConjSDC}. The result follows.
\end{proof}

It turns out that the converse is also true (cf. Theorem \ref{ThmEquiv}), but the proof is more delicate. 

\subsection{Preliminary lemmas}
\begin{lemma}\label{LemmaDet} Let $K$ be a field and let $m<n$ be positive integers. Let ${\bf v}_i=(v_{i,1},...,v_{i,n})\in K^n$ for $1\le i\le m$ be linearly independent over $K$. Let $j_0$ be such that $v_{i,j_0}\ne 0$ for some $i$. There is an injective function $\tau:\{1,...,m\}\to \{1,...,n\}$ such that $j_0$ is in the image of $\tau$ and for each $1\le i\le m$ we have $v_{i,\tau(i)}\ne 0$. 
\end{lemma}
\begin{proof} Let $I=\{1,...,m\}$ and $J=\{1,...,n\}$. Let $A=[v_{i,j}]_{i\in I, j\in J}$ and note that this matrix has rank $m$  by linear independence of its rows. The $j_0$-th column is not the zero vector, so we may choose $J'\subseteq J$ with $\# J'=m$ such that the square matrix $A'=[v_{i,j}]_{i\in I, j\in J'}$ still has rank $m$. In particular, $\det(A')\ne 0$. Writing $\det(A')=\sum_\sigma \pm \prod_{i} v_{i,\sigma(i)}$ where $\sigma$ varies over bijective functions $I\to J'$ (with suitable choice of signs) we see that for some bijective $\tau:I\to J'$ we have $\prod_{i} v_{i,\tau(i)}\ne 0$.
\end{proof}
\begin{lemma}\label{Lemmaprodv} Let $\epsilon>0$. For all but finitely many positive integers $n$ we have $\prod_{p|n}v_p(n) < n^\epsilon$.
\end{lemma}
\begin{proof}   Note that $\prod_{p|n}v_p(n)\le \sigma_0(n)$ where $\sigma_0(n)$ is the number of positive divisors of $n$. Thus, the result  follows from standard bounds on $\sigma_0(n)$.
\end{proof}

We remark that a much more precise version of Lemma \ref{Lemmaprodv} is due to de Weger \cite{deWeger}. 

The following result limits how small $\|\psi\|$ can be when $a,b$ are $\psi$-dependent. Note that the condition that $a,b,c$ are not of the form $1,N,q$ with $q$ prime (up to order) from our heuristic in Section \ref{SecHeuristic}, naturally appears here again.
\begin{lemma}\label{LemmaKey} Let $a,b,c$ be coprime positive integers with $a+b=c$, not of the form $(1,8,9)$ or $(1,N,q)$ with $q$ prime (up to order).  Define $r=\omega(abc)-1$.  Let  $\psi_1,...,\psi_{r-1} \in \Tcal^\circ(a,b)$ be linearly independent derivations; in particular, $a$ and $b$ are $\psi_i$-dependent for each $i$. Suppose that there is some number $M$ satisfying $1<M<2$ and $c<\rad(abc)^M$, and let $\mu= (2-M)/(4M)$. Then 
$$
 \prod_{i=1}^{r-1} \|\psi_i\| \ge \frac{c^{\mu}}{\prod_{p|abc} v_p(abc)}.
$$
\end{lemma}
\begin{proof} Recall that $\Tcal^\circ(a,b)$ is defined by the conditions \eqref{EqnAdd} and \eqref{EqnDep}. Together they give
\begin{equation}\label{EqnDepDep}
\sum_{p|a} \frac{v_p(a)}{p}\psi(\xi_p) =\sum_{p|b} \frac{v_p(b)}{p} \psi(\xi_p) = \sum_{p|c} \frac{v_p(c)}{p} \psi(\xi_p)
\end{equation}
which holds for every $\psi\in \Tcal^\circ(a,b)$, in particular for each $\psi_i$. In fact,  \eqref{EqnAdd} and \eqref{EqnDep} together are equivalent to \eqref{EqnDepDep}, so, 
$$
\Tcal^\circ(a,b)=\{\psi\in \Tcal : \supp(\psi)\subseteq \supp(abc)\mbox{ and \eqref{EqnDepDep} holds}\}.
$$ 

 We distinguish three cases (Lemma \ref{LemmaCatalan} and our assumptions  imply that there is no other case):
\begin{itemize}
\item[(i)]  Both $ab$ and $c$ have at least $2$ different prime factors each. 
\item[(ii)]  Up to order, we have $(a,b,c)=(1,q^s,N)$ for a prime $q$ and some integer $s\ge 2$ and $N$ with at least two prime factors.
\item[(iii)] $(a,b,c)=(q_1^{s_1}, q_2^{s_2}, q_3^{s_3})$ where  $q_1, q_2, q_3$ are different primes and $s_i\ge 1$ for each $i$.  
\end{itemize}
 
Let us first deal with cases (i) and (ii).

In case (i), suppose that there is some prime $q|abc$ such that $\psi_i(\xi_q)=0$ for each $i$. Then every $\psi\in \Tcal^\circ(a,b)$ would satisfy $\psi(\xi_q)=0$, because the derivations $\psi_1,...,\psi_{r-1}$ generate a finite index subgroup of $\Tcal^{\circ}(a,b)$ (cf. Lemma \ref{LemmaDep}). This is not possible, since the condition $\psi(\xi_q)=0$ is linearly independent from the two equations in \eqref{EqnDepDep} that define $\Tcal^\circ(a,b)$. This proves that in case (i), for each prime $p|abc$ we have $(\psi_i(\xi_p))_i\ne (0,...,0)$.

In case (ii) we note that one of the equations in \eqref{EqnDepDep} is $0=s\psi(\xi_q)/q$, which is equivalent to $\psi(\xi_q)=0$. Therefore, $\Tcal^0(a,b)$ is defined by $\psi(\xi_q)=0$ and $\sum_{p|N} v_p(N)\psi(\xi_p)/p=0$. This last equation is linearly independent from any condition of the form $\psi(\xi_p)=0$ with $p\ne q$ because $N$ has at least two prime factors. This proves that in case (ii), for each $p|abc$ with $p\ne q$ we have $(\psi_i(\xi_p))_i\ne (0,...,0)$.

Let $q'$ be the largest prime factor of $abc$ in case (i), and let it be the largest prime factor of $abc$ subject to the condition $q'\ne q$ in case (ii). In either case, $(\psi_i(\xi_{q'}))_i\ne (0,...,0)$.

Let $I=\{1,...,r-1\}$ and $J=\{ p : p|abc\}$, so that $\#I=r-1 < \#J=r+1$. Choosing the vectors ${\bf v}_i = (\psi_i(\xi_p))_{p\in J}$ for $i\in I$, Lemma \ref{LemmaDet} gives an injective function $\tau:I\to J$ having $q'$ in its image such that for every $i\in I$ we have $\psi_i(\xi_{p_i})\ne 0$ where $p_i:=\tau(i)$.

By coprimality of $a,b,c$ and considering the denominators in \eqref{EqnDepDep}, we see that for each $p|abc$ and each $i$ we have that $p$ divides $v_p(abc)\psi_i(\xi_p)$. Together with the previous non-vanishing, for each $i=1,...,r-1$ we find $v_{p_i}(abc)\|\psi_i\| \ge p_i$. This gives
$$
\prod_{p|abc} v_p(abc)\cdot \prod_{i=1}^{r-1}\|\psi_i\|\ge\prod_{i=1}^{r-1}(v_{p_i}(abc)\|\psi_i\| )\ge \prod_{i=1}^{r-1} p_i =Pq'
$$
where $P$ is the product of the primes $p_i\ne q'$. Let $\ell_1,\ell_2\in J$ be the two primes not in the image of $\tau$. Then $\rad(abc)=P\ell_1\ell_2q'$.

In case (i) we have $\ell_1,\ell_2<q'$ so, $\rad(abc)=P\ell_1\ell_2q'< P\cdot (q')^3\le (Pq')^3$. This proves $\prod_{i=1}^{r-1} p_i \ge \rad(abc)^{1/3}\ge c^{1/(3M)}$, which concludes the proof in case (i). 

In case (ii) notice that $q=\ell_j$ for $j=1$ or $j=2$. Let us assume $q=\ell_1$, in particular, $\ell_2<q'$. Observe that  $\ell_1^2=q^2\le q^s\le c$, so $\ell_1\le c^{1/2}$. Then we get 
$$
c^{1/M}\le \rad(abc)=P\ell_1\ell_2 q' \le P(q')^2c^{1/2} \le (Pq')^2c^{1/2}.
$$
This proves $\prod_{i=1}^{r-1} p_i \ge c^{(2-M)/(4M)}$, which concludes the proof in case (ii). 

Finally, let us consider case (iii). Naturally, one of the primes $q_i$ is $2$ but this will not be relevant. Note that $r=2$, so we need a lower bound for $\|\psi_1\|$. By \eqref{EqnDepDep} we find $s_1\psi_1(\xi_{q_1})/q_1=s_2\psi_1(\xi_{q_2})/q_2=s_3\psi_1(\xi_{q_3})/q_3$ and it follows that $\rad(abc)=q_1q_2q_3$ divides $s_1s_2s_3\psi_1(\xi_{q_1})\psi_1(\xi_{q_2})\psi_1(\xi_{q_3})$. In particular 
$$
\|\psi_1\|^3\cdot\prod_{p|abc} v_p(abc)^3\ge \|\psi_1\|^3\cdot \prod_{p|abc} v_p(abc)\ge \rad(abc)>c^{1/M}
$$
which gives the result in case (iii).
\end{proof}

\subsection{The $abc$ Conjecture implies the Small Derivatives Conjecture}
\begin{theorem}\label{ThmEquiv} If Oesterl\'e's $abc$ Conjecture \ref{ConjWeakabc} holds with some exponent $1<M<2$, then the Small Derivatives Conjecture  \ref{ConjSDC} holds for each exponent $\eta> 1-(2-M)/(4M)$.
\end{theorem}
Let us remark  that for $1<M<2$ the quantity $\mu=(2-M)/(4M)$ satisfies $3/4<1-\mu < 1$. We see that any exponent $\eta>1-\mu$ sufficiently close to $1-\mu$ satisfies $\eta<1$, hence, it is admissible for the Small Derivatives Conjecture  \ref{ConjSDC}.
\begin{proof}[Proof of Theorem \ref{ThmEquiv}] We assume that Oesterl\'e's $abc$ Conjecture \ref{ConjWeakabc} holds for some exponent $M$ with $1<M<2$. Let us fix $\epsilon>0$. In the argument below, we may need to implicitly discard finitely many triples $(a,b,c)$ for some inequalities to hold, which we indicate by writing ``$\le_*$'' instead of ``$\le$''. The finitely many discarded triples will only depend on $M$ and $\epsilon$.

Let $a,b$ be coprime positive integers, set $c=a+b$, and assume that $(a,b,c)$ is not of the form $(1,N,q)$ with $q$ prime, up to order. 

Let $\psi_1,...,\psi_r\in \Tcal(a,b)$ be as provided by Theorem \ref{ThmExistence} and label these derivations in such a way that $\|\psi_1\|\le \|\psi_2\|\le ...\le \|\psi_r\|$. Let $i_0\in\{1,2,...,r\}$ be the least index such that $\psi_{i_0}\notin \Tcal^\circ(a,b)$, which exists by Lemma \ref{LemmaDep}. We distinguish two cases:

\begin{itemize}
\item[(a)] $i_0<r$. In this case, using Lemma \ref{Lemmaw} we get $\|\psi_{i_0}\|\le_* c^{(1+\epsilon)/2}$ because
$$
\|\psi_{i_0}\|^{2}\le \prod_{i=i_0}^r\|\psi_{i}\|\le \frac{\omega(abc)}{2\log 2} c\log c\le_* c^{1+\epsilon}.
$$

\item[(b)] $i_0=r$. In this case we have $\psi_1,...,\psi_{r-1}\in \Tcal^\circ(a,b)$ and we can apply Lemma \ref{LemmaKey} because we are assuming Conjecture \ref{ConjWeakabc} for some exponent $1<M<2$. Let us define $\mu=  (2-M)/(4M)$.  Lemmas \ref{Lemmaw} and \ref{Lemmaprodv} give $\|\psi_r\| \le_* c^{1-\mu + \epsilon}$ because
$$
c^{\mu-\epsilon/2} \cdot \|\psi_r\| \le_* \frac{c^\mu}{\prod_{p|abc} v_p(abc)} \cdot \|\psi_r\| \le \prod_{i=1}^r \|\psi_i\|\le \frac{\omega(abc)}{2\log 2} c\log c\le_* c^{1+\epsilon/2}.
$$

\end{itemize}
The second case is the one giving the worst bound, hence the result.
\end{proof}

In particular, Lemma \ref{LemmaEquiv} and Theorem \ref{ThmEquiv} give:

\begin{corollary}\label{CoroEquiv} The Masser-Oesterl\'e $abc$ Conjecture \ref{Conjabc} implies the Small Derivative Conjecture \ref{ConjSDC}. Conversely, the Small Derivative Conjecture \ref{ConjSDC} implies Oesterl\'e's $abc$ Conjecture \ref{ConjWeakabc}.
\end{corollary}


\section{Differentials of rings over monoids}\label{SecFinal}

\subsection{Definitions} Let $A$ be a commutative unitary ring, let $R$ be a commutative monoid, and let $\alpha:R\to A$ be a morphism of monoids with $A$ taken as a multiplicative monoid.  Given an $A$-module $U$, a \emph{$U$-valued $\alpha$-derivation on $A$} is a function $D:A\to U$  satisfying
\begin{itemize}
\item[(Diff1)] $R$-triviality: $D(\alpha(r))=0$ for all $r\in R$
\item[(Diff2)] Leibniz rule: $D(ab)=aD(b)+bD(a)$ for all $a,b\in A$.
\end{itemize}

A \emph{differential $(A,\alpha)$-module}  is a pair $(U,D)$ where $U$ is an $A$-module and $D$ is a $U$-valued $\alpha$-derivation on $A$. 

Naturally, these definitions can also be formulated when $A$ is just assumed to be a commutative monoid, which is perhaps better suited for the theory of monoid schemes (``geometry over $\F_1$'', cf. \cite{Deitmar}). However, we keep the assumption that $A$ be a ring to simplify the exposition and because this is the case of interest for us. Another observation is that when $R=\{1\}$ we recover the notion of \emph{absolute derivation} from \cite{KOW} and, in fact, most of that theory can be generalized to our setting.

One directly checks
\begin{lemma} Let $(U,D)$ be a differential $(A,\alpha)$-module. We have:
\begin{itemize}
\item[(i)] $D(0)=D(1)=0$.
\item[(ii)] For all $r\in R$ and $b\in A$ we have $D(\alpha(r)b)=\alpha(r)D(b)$.
\item[(iii)] Given $a\in A$ and a positive integer $n$, we have $D(a^n)=na^{n-1}D(a)$.
\item[(iv)] Given $u\in A^\times$ and a positive integer $n$, we have $D(u^{-n})=-nu^{-(n+1)}D(u)$.
\end{itemize}
\end{lemma}

Given differential $(A,\alpha)$-modules $(U,D)$ and $(V,E)$, a morphism of differential $(A,\alpha)$-modules is a morphism of $A$-modules $f:U\to V$ that satisfies $E=f\circ D$. We obtain a category of differential $(A,\alpha)$-modules which we denote by $\Phi_{(A,\alpha)}$.

For an $A$-module $U$, let $\Der_{(A,\alpha)}(U)=\{D:A\to U \, : (U,D)\in Ob(\Phi_{(A,\alpha)})\}$. This is an $A$-module with the structure induced by $U$.  Given $A$-modules $U$ and  $V$  and a morphism $f\in \Hom_A(U,V)$, we define $\Der_{(A,\alpha)}(f):\Der_{(A,\alpha)}(U)\to \Der_{(A,\alpha)}(V)$ by $\Der_{(A,\alpha)}(f)(D)=f\circ D$. 
\begin{lemma}
The rule $\Der_{(A,\alpha)}$ defines a functor $A\mbox{-}\mathbf{Mod}\to A\mbox{-}\mathbf{Mod}$.
\end{lemma}

\subsection{Universal object} Consider $\alpha:R\to A$ as before. Let $X_A$ be the free $A$-module on the generators $e_a$ for $a\in A$. Let $M_{(A,\alpha)}\subseteq X_A$ be the sub $A$-module generated by the elements $e_{\alpha(r)}$ for $r\in R$ and $e_{ab}-ae_b-be_a$ for $a,b\in A$. We consider the quotient $A$-module $\Omega_{(A,\alpha)}=X_A/M_{(A,\alpha)}$ and define $\dd_{(A,\alpha)}:A\to \Omega_{(A,\alpha)}$  by $\dd_{(A,\alpha)} (a)= e_a\bmod M_{(A,\alpha)}$. By construction, $(\Omega_{(A,\alpha)},\dd_{(A,\alpha)})$ is a differential $(A,\alpha)$-module. If there is no risk of confusion, we will simply write $\dd$ instead of $\dd_{(A,\alpha)}$.

\begin{lemma}[Universal property of $\Omega_{(A,\alpha)}$]  For each $A$-module $U$, the rule $\psi\mapsto \psi\circ \dd$ defines a functorial isomorphism of $A$-modules $\eta_{U}:\Hom_A(\Omega_{(A,\alpha)}, U)\to \Der_{(A,\alpha)}(U)$.  Thus, $\Omega_{(A,\alpha)}$ represents the functor $\Der_{(A,\alpha)}$. In particular, $(\Omega_{(A,\alpha)},\dd)$ is an initial object in the category $\Phi_{(A,\alpha)}$.
\end{lemma}
\begin{proof} Functoriality on $U$ and $A$-linearity are immediate. Let us check that $\eta_U$ is an isomorphism.

Let $\psi\in \Hom_A(\Omega_{(A,\alpha)},U)$ with $\eta_U(\psi)=0$. This means that $\psi\circ \dd : A\to U$ is the zero map. The set $\dd(A)$ generates $\Omega_{(A,\alpha)}$ as an $A$-module, so $\psi=0$ because it vanishes on a generating set of  $\Omega_{(A,\alpha)}$. Thus, $\eta_U$ is injective.

Let $D\in \Der_{(A,\alpha)}(U)$. Let $\theta:X_A\to U$ be the $A$-module map determined by $\theta(e_a)=D(a)$ on the standard basis $\{e_a\}_{a\in A}$ of the free $A$-module $X_A$. Let $\tilde{\dd}:A\to X_A$ be the function $\tilde{\dd}(a)=e_a$ and let $\pi:X_A\to X_A/M_{(A,\alpha)}=\Omega_{(A,\alpha)}$ be the quotient map. Note that $\theta\circ \tilde{\dd}=D$ and $\dd=\pi\circ \tilde{\dd}$. Since $D$ satisfies (Diff1) and (Diff2), we have that a generating set for $M_{(A,\alpha)}$ is contained in $\ker(\theta)$, and since $\theta$ is $A$-linear we get $M_{(A,\alpha)}\subseteq \ker(\theta)$. Thus, there is an $A$-module map $\psi:\Omega_{(A,\alpha)}\to U$ with $\theta=\psi\circ\pi$. Therefore, $D=\theta\circ \tilde{\dd} = \psi\circ \pi\circ \tilde{\dd} = \psi\circ \dd =\eta_U(\psi)$, proving that  $\eta_U$ is surjective.
\end{proof}

We call $(\Omega_{(A,\alpha)},\dd)$ the \emph{universal} differential $(A,\alpha)$-module.

\subsection{Examples} We conclude by discussing some concrete examples.

\begin{example} Let $A=\F_q$ be a finite field with $q$ elements and $\alpha:R\to \F_q$ be arbitrary. The elements $\dd(x)$ for $x\in \F_q$ generate  $\Omega_{(\F_q,\alpha)}$, and $\dd(x)=\dd(x^q)=qx^{q-1}\dd(x)=0$. Therefore, $\Omega_{(\F_p,\alpha)}=(0)$.
\end{example}
\begin{example} Let $A=\Z/4\Z$ and let $\alpha:\{1\}\to \Z/4\Z$ be the inclusion. In this case it is not so lengthy to directly compute $M_{(A,\alpha)}\subseteq X_A=(\Z/4\Z)^4$. One finds that the universal $\alpha$-derivation is $\dd:\Z/4\Z\to \Z/2\Z\oplus \Z/2\Z$ defined by $\dd(0)=\dd(1)=(0,0)$, $\dd(2)=(1,0)$, and $\dd(3)=(0,1)$. Note that $ \dd(1)+\dd(2) = (1,0)\ne (0,1) =\dd(3)$, so, $\dd$ is not additive. Nevertheless, let $\sigma:(\Z/2\Z)^2\to \Z/2$ be $\sigma(x,y)=x+y$. Then the $\alpha$-derivation $\sigma\circ \dd:\Z/4\Z\to \Z/2\Z$ respects the equation $1+2=3$.
\end{example}
\begin{example} Let $A$ be a UFD and let $T$ be a set of pairwise non-associated irreducible elements. Let $R=A - \cup_{t\in T} (t)$, let $\alpha:R\to A$ be the inclusion, and let $U=\bigoplus_{t\in T} A$. Define $D:A\to U$ by $D(a)=(v_t(a)\cdot at^{-1})_{t\in T}$  where $v_t$ is the $t$-adic valuation. Then $D:A\to U$ is an $\alpha$-derivation and we claim it is the universal one. Indeed, given $a=rt_1^{n_1}\cdots t_k^{n_k}\in A$ with $r\in R$, $n_j\ge 1$, and $t_j\in T$ different, the map $\dd=\dd_{(A, \alpha)}$ satisfies $\dd(a)=\sum_{j=1}^k n_j at_j^{-1}\dd(t_j)$. Since $U$ is free, there is an $A$-module map $\phi: U\to \Omega_{(A,\alpha)}$ satisfying $\dd=\phi\circ D$. We conclude by universality of $\Omega_{(A,\alpha)}$.
\end{example}
\begin{example} In the previous example, consider the special case $A=\Z$ and $T$ the set of all prime numbers, so that $R=\{-1,1\}$. Then $D:A\to U$ turns out to be our map $\dd:\Z\to \Omega$. So,  the latter  is the universal $\alpha$-derivation  when $\alpha:\{-1,1\}\to \Z$ is the inclusion. Thus, $\Hom_\Z(\Omega,\Z)\simeq \Der_{(\Z,\alpha)}(\Z)$ is the module of all $\alpha$-derivations $D:\Z\to \Z$. Our $\Z$-module $\Tcal$  is a metrized version of this.
\end{example}

%
\section{Acknowledgments}

This research was supported by ANID (ex CONICYT) FONDECYT Regular grant 1190442 from Chile. 

The initial motivation for this project was a conversation with Thanases Pheidas that took place at the 2016 Oberwolfach workshop \emph{Definability and Decidability Problems in Number Theory}. I heartily thank the MFO for their support and hospitality, as well as T. Pheidas for bringing  the topic of arithmetic derivatives to my attention. 

Comments by Jerson Caro and Natalia Garcia-Fritz on a first version of this manuscript are gratefully acknowledged. I also thank the referee for carefully reading this article and for  valuable suggestions and corrections.



\begin{thebibliography}{9}      



\bibitem{Barbeau} E. Barbeau, \emph{Remarks on an arithmetic derivative}. Canadian Mathematical Bulletin. 4 (2): 117-122 (1961).

\bibitem{BombieriVaaler} E. Bombieri, J. Vaaler, \emph{On Siegel's lemma}. Invent. Math. 73 (1983), no. 1, 11-32.

\bibitem{Buium} A. Buium, \emph{Arithmetic differential equations}. Mathematical Surveys and Monographs, 118. American Mathematical Society, Providence, RI, 2005. xxxii+310 pp. ISBN: 0-8218-3862-8 

\bibitem{Deitmar} A. Deitmar, \emph{Schemes over $\F_1$}. Number fields and function fields---two parallel worlds, 87-100, Progr. Math., 239, Birkh\"auser Boston, Boston, MA, 2005.

\bibitem{FaltingsKS} G. Faltings, \emph{Does there exist an arithmetic Kodaira-Spencer class?} Algebraic geometry: Hirzebruch 70 (Warsaw, 1998), 141-146,  Contemp. Math., 241, Amer. Math. Soc., Providence, RI, 1999. 


\bibitem{KOW} N. Kurokawa, H. Ochiai, M. Wakayama, \emph{Absolute derivations and Zeta functions.} Kazuya Kato's fiftieth birthday. Doc. Math. 2003, Extra Vol., 565-584. 


\bibitem{Masserabc} D. Masser, \emph{Abcological anecdotes}. Mathematika 63 (2017), no. 3, 713-714. 

\bibitem{Mihailescu} P. Mihailescu, \emph{Primary cyclotomic units and a proof of Catalan's conjecture}. (English summary) J. Reine Angew. Math. 572 (2004), 167-195.

\bibitem{MingotShelly} J. Mingot Shelly, \emph{Una cuesti\'on de la teor\'ia de los n\'umeros}. Asociation Esp. Granada: 1-12 (1911).

\bibitem{Snyder} N. Snyder,  \emph{An alternate proof of Mason's theorem}. Elem. Math. 55 (2000), no. 3, 93-94.

\bibitem{VojtaThesis} P. Vojta, \emph{Diophantine Approximations and Value Distribution Theory}. Lecture Notes in Math., vol. 1239, Springer-Verlag, Berlin, 1987.

\bibitem{VojtaCIME} P. Vojta, \emph{Diophantine approximation and Nevanlinna theory}. Arithmetic geometry, 111-224, Lecture Notes in Math., 2009, Springer, Berlin, 2011. 

\bibitem{deWeger} B. de Weger, \emph{$A+B=C$  and big $\Sha$'s}.  
Quart. J. Math. Oxford Ser. (2) 49 (1998), no. 193, 105-128.

\bibitem{Wiles} A. Wiles, \emph{Modular elliptic curves and Fermat's last theorem}. Ann. of Math. (2) 141 (1995), no. 3, 443-551.

\end{thebibliography}
\end{document}